\documentclass[12pt, reqno]{amsart}

\usepackage{amssymb, amsmath, amsthm}
\usepackage[backref]{hyperref}
\usepackage[alphabetic,backrefs,lite]{amsrefs}
\usepackage{amscd}   
\usepackage{fullpage}
\usepackage[all]{xy} 

\DeclareFontEncoding{OT2}{}{} 

\usepackage{color}

\newtheorem{lemma}{Lemma}[section]
\newtheorem{theorem}[lemma]{Theorem}

\newtheorem{prop}[lemma]{Proposition}
\newtheorem{cor}[lemma]{Corollary}

\newtheorem{claim*}{Claim}

\newtheorem{remark}[lemma]{Remark}
\newtheorem{thm}[lemma]{Theorem}
\newtheorem{defn}[lemma]{Definition}
\newtheorem{example}[lemma]{Example}
\newtheorem{notation}[lemma]{Notation}
\newtheorem{question}[lemma]{Question}

\newcommand{\PP}{{\mathbb P}}
\newcommand{\C}{{\mathbb C}}
\newcommand{\F}{{\mathbb F}}
\newcommand{\Q}{{\mathbb Q}}
\newcommand{\R}{{\mathbb R}}
\newcommand{\Z}{{\mathbb Z}}

\newcommand{\pp}{{\mathfrak p}}
\newcommand{\qq}{{\mathfrak q}}
\newcommand{\mm}{{\mathfrak m}}

\newcommand{\calO}{{\mathcal O}}

\newcommand{\vareps}{{\varepsilon}}

\DeclareMathOperator{\Char}{char}

\DeclareMathOperator{\Hom}{Hom}

\DeclareMathOperator{\Gal}{Gal}

\DeclareMathOperator{\Nm}{Nm}

\DeclareMathOperator{\ord}{ord}

\DeclareMathOperator{\red}{red}

\newcommand{\idele}{id\`ele }

\numberwithin{equation}{section}
\numberwithin{table}{section}

\title{A universal first order formula defining the ring of integers in a number field}
\subjclass[2010]{Primary 11R37; Secondary 11R52, 11U05}
\keywords{Hilbert's Tenth Problem, Diophantine set, quaternion algebra, class field theory, Artin reciprocity, Hilbert symbol}
\author{Jennifer Park}
\curraddr{Department of Mathematics, Massachusetts Institute of Technology, Cambridge, MA 02139-4307, USA}
\email{jmypark@math.mit.edu}
\urladdr{http://math.mit.edu/~jmypark}
\date{\today}

\begin{document}
\begin{abstract}
We show that the complement of the ring of integers in a number field $K$ is Diophantine. This means the set of ring of integers in $K$ can be written as $\{t \in K \mid \forall x_1, \cdots, x_N \in K, f(t,x_1, \cdots, x_N) \neq 0\}$. We will use global class field theory and generalize the ideas originating from Koenigsmann's recent result giving a universal first order formula for $\Z$ in $\Q$.
\end{abstract}
\maketitle
\section{Introduction}

Hilbert's tenth problem asked for an algorithm that decides whether an integer solution to a polynomial equation $f(x_1, \cdots, x_n) = 0$ exists, given $f \in \Z[x_1, \cdots, x_n]$. Matiyasevich answered the question in the negative, building on earlier results by Davis, Putnam and Robinson, by showing the equivalence of Diophantine sets in $\Z$ and listable sets in $\Z$; a set $A \subseteq \Z$ is said to be \textit{Diophantine} if there exists a polynomial $f \in \Z[t,x_1, \cdots, x_n]$ such that $A = \{t \in \Z \mid \exists x_1, \cdots, x_n, f(t,x_1, \cdots, x_n) = 0\}$, and $A$ is \textit{listable} if there is an algorithm that eventually prints all the elements and only the elements of $A$. With this equivalence, and the undecidability of the halting problem, one easily finds a listable set $A$ whose membership cannot be determined by any algorithm. By Matiyasevich's theorem, $A$ is also a Diophantine set, defined by some polynomial $f(t, x_1, \cdots, x_n)$. Then there is no algorithm for deciding which polynomials of the form $f(a, x_1, \cdots, x_n)$ with $a \in \Z$ has integer solutions.

Matiyasevich's work leads to the natural extensions of Hilbert's tenth problem to other settings where the notion of listable sets makes sense. More precisely stated, we ask:

\begin{question}
Let $R$ be a commutative and countable ring, with a fixed representation of elements of $R$ by integers. Given a polynomial $f \in R[x_1, \cdots, x_n]$, is there an algorithm that decides whether $f$ has a solution in $R$?
\end{question}

One possibility for $R$ is the ring of integers $\calO_K$ of a number field $K$. In this setting, the above problem is open, although in many cases, it has been answered in the negative. Assuming the Shafarevich-Tate conjecture, it is in fact shown by Mazur and Rubin in \cite{MazRub10} that Hilbert's tenth problem has a negative answer over the ring of integers of every number field.

We can also consider Hilbert's tenth problem over $\Q$, or, even more generally, over a number field $K$. This problem is of interest, because of its equivalence to the big problem of arithmetic geometry asking for an algorithm for deciding the existence of a rational point on a variety defined over $K$. This problem remains open. However, we can attempt to perform a reduction via the following method: if $\calO_K$ were a Diophantine set in $K$ (where the notion of a Diophantine set is defined analogously as before, by replacing $\Z$ by $K$), then given a polynomial over $K$, we could impose the extra condition that all the variables actually take values in $\calO_K$, and now the answer to Hilbert's tenth problem over $K$ depends on the answer to Hilbert's tenth problem over $\calO_K$.

It is unknown whether $\calO_K$ is Diophantine in $K$, even in the case $K = \Q$. Robinson \cite{Rob49} found a first-order definition of $\Z$ in $\Q$, given by an $\forall \exists \forall$-formula. Recently, Poonen \cite{Poo09} gave an $\forall \exists$-definition of $\calO_K$ in $K$, and Koenigsmann \cite{Koe10} extended this result to give an $\forall$-definition of $\Z$ in $\Q$. In this paper, we generalize Koenigsmann's results to the setting of number fields to give an $\forall$-definition of $\calO_K$ in $K$. We prove the following theorem:
\begin{thm}
There is a first-order universal formula defining $\calO_K$ in $K$. That is, $K - \calO_K$ is Diophantine in $K$.
\end{thm}

Our proof uses many ideas from \cite{Koe10}, including the use of quaternion algebras, and the fact that a Diophantine definition of the Jacobson radical of a ring implies universal first-order definition of a related ring. 

A key construction in \cite{Koe10} in obtaining a universal first-order formula for $\Z$ in $\Q$ is to split up the odd prime numbers into four sets, depending on their values mod $8$. This construction does not at all generalize to the setting of $\calO_K$ in $K$ for several reasons. The main obstruction comes from the fact that in general, one cannot expect a prime ideal to be principal. Thus, it is no longer possible to split up the prime ideals based on simple modular arithmetic. Further, even in the simplified cases of the class number of $K$ being one, the number $8$ does not have a natural interpretation; no straightforward generalizations of \cite{Koe10} seem to exist.

In order to clear up these issues, we introduce global class field theory and the notion of ray classes, which replaces the congruence classes modulo $8$. This offers a simpler and a more natural alternative to the construction in \cite{Koe10}, and this approach indeed generalizes to all number fields. The proofs that follow are further aided by a theorem, proved and communicated by Tate, which allows one to find an element $x \in K$ that has prescribed Hilbert symbols against finitely many elements of $K$. The special case where $K = \Q$ is well-known; for example, see \cite{Ser73}, Theorem 4, page 24, but the general case of $K$ being a global field does not appear in the literature.

Section 2 is based on the ideas from \cite{Poo09}, and Section 3 and 4 are generalizations of \cite{Koe10}.


\section{A Universal-Existential Definition of \texorpdfstring{$\calO_K$}{OK} in \texorpdfstring{$K$}{K}}

Throughout, $K$ is a fixed number field, $\calO_K$ denotes its ring of integers. For a prime ideal $\pp \subseteq \calO_K$, and its associated valuation $v$, $\calO_{v}$ is the set of the elements of the completion $K_{v}$ whose valuations are nonnegative. We denote $(\calO_K)_{\pp}$ to be the localization of the ring of integers of $K$ at the prime $\pp$. Then $(\calO_K)_{\pp} \subset \calO_v$.

\begin{notation}
Let $\PP$ be the set of all finite places of $K$, and let $\PP \cup \infty$ be the set of all places of $K$, both finite and infinite. Prime ideals and their corresponding valuations are used interchangeably. Further, for $a,b \in K^{\times}$,
\begin{itemize}
\item $H_{a,b} := K \cdot 1 \oplus K \cdot \alpha \oplus K \cdot \beta \oplus K \cdot \alpha \beta$ is the quaternion algebra over $K$ with multiplication defined by $\alpha^2 = a, \beta^2 = b$, and $\alpha \beta = -\beta \alpha$.
\item $\Delta_{a,b} := \{v \in \PP \cup \infty \mid H_{a,b} \otimes K_v \text{ does not split}\}$. We note that $\Delta_{a,b}$ is always finite.
\item $S_{a,b} := \{2x_1 \in K \mid \exists x_2, x_3, x_4 \in K \text{ with } x_1^2 - ax_2^2 - bx_3^2 + abx_4^2 = 1\}$ is the set of traces of norm-$1$ elements of $H_{a,b}$.
\item $T_{a,b} := S_{a,b} + S_{a,b}$.
\end{itemize}
We note in particular that $S_{a,b}$ and $T_{a,b}$ are Diophantine.
\end{notation}

For each place $v$ of $K$, we can similarly define $S_{a,b}(K_v)$ and $T_{a,b}(K_v)$ by replacing $K$ by $K_v$. For each infinite place $\sigma$ we define
$$ \calO_{\sigma} :=
 \begin{cases}
 \mathbb{R}, & \mbox{if } \sigma \mbox{ is a real place, and } \sigma(a)>0 \mbox{ or } \sigma(b)>0\\
 [-4,4], & \mbox{if } \sigma \mbox{ is a real place, and } \sigma(a),\sigma(b) < 0 \\
 \mathbb{C}, & \mbox{if } \sigma \mbox{ is a complex place.}
 \end{cases}$$

Let $v$ be a finite place of $K$, and let $\F_v$ be the residue field of $v$ of size $q$, which is some power of a prime $p$. Then denote the reduction map by $\red_v \colon \calO_v \to \F_v$. Further, define
$$U_v:=\{s \in \F_v \mid x^2 - sx + 1 \mbox{ is irreducible over } \F_v\}.$$

\begin{lemma}
\label{L:completion}
\hfill
\begin{itemize}
\item[(a)] If $v \notin \Delta_{a,b},$ then $S_{a,b}(K_v) = K_v$.
\item[(b)] If $v \in \Delta_{a,b} \cap \PP,$ then $\red_{v}^{-1}(U_v) \subseteq S_{a,b}(K_v) \subseteq (\calO_K)_v$.
\item[(c)] For an infinite place $\sigma$,
$$ S_{a,b}(K_{\sigma}) =
 \begin{cases}
 \mathbb{R}, & \mbox{if } \sigma \mbox{ is a real place, and } \sigma(a)>0 \mbox{ or } \sigma(b)>0\\
 [-2,2], & \mbox{if } \sigma \mbox{ is a real place, and } \sigma(a),\sigma(b) < 0 \\
 \mathbb{C}, & \mbox{if } \sigma \mbox{ corresponds to a complex embedding.}
 \end{cases}$$
 \item[(d)] For any $v$ with $\# \F_v > 11$, we have $\F_v = U_v + U_v$.
 \item[(e)] For each $a, b \in K^{\times}$ such that $\sigma(a)>0$ or $\sigma(b)>0$ for each real archimedean place $\sigma$,
 $$S_{a,b} = K \cap \bigcap_{v \in \Delta_{a,b}} S_{a,b}(K_v).$$
\end{itemize}
\end{lemma}
\begin{proof}
\hfill
\begin{itemize}
\item[(a)] It is clear from the definition of $S_{a,b}$ that $S_{a,b}(K_v) \subseteq K_v$. Now, to prove the reverse inclusion, take any element $s \in K_v$. We will show that there is an element in the quaternion algebra $H_{a,b}$ over $K_v$ whose reduced trace is $s$. Since $v \notin \Delta_{a,b}$, we have $H_{a,b} \otimes K_v \cong M_2(K_v)$, so every monic quadratic polynomial is a characteristic polynomial of some element of the matrix ring. In particular, $K_v \subseteq S_{a,b}(K_v)$, which shows the equality of part (a).
\item[(b)] See \cite{Poo09}, Lemma 2.1.
\item[(c)] The first and third cases are handled by (a). The second case is a straightforward computation.
\item[(d)] See \cite{Poo09}, Lemma 2.3.
\item[(e)] This is a special case of the Hasse-Minkowski local-global principle.
\end{itemize}
\end{proof}

\begin{prop} 
\label{P: Tab}
For any $a,b \in K^{\times}$ such that $\sigma(a)>0$ or $\sigma(b)>0$ for each real archimedean place $\sigma$,
$$T_{a,b} = \bigcap_{\pp \in \Delta_{a,b}}\calO_{\pp}.$$
\end{prop}
\begin{proof}

Let $T_{a,b}'$ be the right-hand side. By Lemma \ref{L:completion} (b) and (e), we have $S_{a,b} \subseteq T_{a,b}'$, so $T_{a,b} \subseteq T_{a,b}'$.

To prove the converse inclusion, we first compute $U_v$ for $\#\F_v<11$. Since $U_v$ only depends on $\F_v$, we may write $U_v = U_{q_v}$, where $q_v = \#\F_v$. We get:
\begin{eqnarray*}
U_2 &=& \{1\}\\
U_3 &=& \{0\}\\
U_4 &=& \{a, a+1\}, \text{ where }a^2 + a + 1 = 0\\
U_5 &=& \{1,4\}\\
U_7 &=& \{0,3,4\}\\
U_8 &=& \{1, a, a^2, a^2 + a\}, \text{ where } a^3 + a + 1 = 0\\
U_9 &=& \{a, a+2, 2a, 2a+1\}, \text{ where }a^2 + 1 = 0\\
U_{11} &=& \{0,1,5,6,10\}.
\end{eqnarray*}

We have $\pm 2 \in S_{a,b}(K_v)$, since $\pm 2$ is the reduced trace of $\pm 1$.
So for each finite place $v$, define $V_v \subseteq \calO_v$ as follows:
$$ V_v =
 \begin{cases}
 \red_v^{-1}(U_{q_v}) \cup \{\pm 2\} & \mbox{if } v|p, 2 \leq p \leq 11\\
 \red_v^{-1}(U_{q_v}) & \mbox{if } v|p, p>11.
 \end{cases}$$
Then by the discussions in the previous paragraph, $V_v \subseteq S_{a,b}(K_v)$, and a case-by-case check on each $2 \leq q_v \leq 11$ shows that
$$(U_{q_v} \cup \{\pm 2\}) + U_{q_v} = \F_v,$$
so $V_v + V_v = \calO_v$ for $v$ with $2 \leq q_v \leq 11$. If $v$ is such that $q_v > 11$, then by Lemma \ref{L:completion}(d), $V_v + V_v = \calO_v$. 

So let $t \in T_{a,b}'$. Then for each $v \in \Delta_{a,b}$, we may choose $r_v \in \calO_v$ such that $r_v, t-r_v \in V_v$. Since $\Delta_{a,b}$ is finite, we use strong approximation to find $r \in \calO$ such that for all $v \in \Delta_{a,b}$, we have $r, t-r \in V_v$. Then by Lemma \ref{L:completion}(e), we have $r, t-r \in S_{a,b}$, which proves the inclusion $T_{a,b}' \subseteq T_{a,b}$.
\end{proof}


\section{Consequences Arising from Global Class Field Theory}

\subsection{Background: Hilbert symbols and class field theory}

In \cite{Ser79}, Corollary to Lemma XIV.3.2, an explicit formula for Hilbert symbols is given: For a number field $K$ and a finite place $v = v_{\pp}$ not lying above $2$, 
\begin{eqnarray}
\label{E: serre}
(a,b)_v &=& \left((-1)^{v(a)v(b)} \red_v \left(\frac{a^{v(b)}}{b^{v(a)}}\right)\right)^{\frac{q-1}{2}},
\end{eqnarray}
where $q = \# \F_v$. Then for a $\pp$-adic unit $a$,
$$(a,p)_v = -1 \Leftrightarrow v(p) \textup{ is odd, and } \red_v(a) \textup{ is not a square in the residue field } \F_{v}.$$

Also, we make the following observation:
\begin{eqnarray*}
\bar{a} \in \F_{v_{\pp}}^2 &\Leftrightarrow & a \in K_v^2 \text{ (Hensel's lemma)}\\
&\Leftrightarrow & \pp_v \textup{ splits in } K(\sqrt a)/K,
\end{eqnarray*}
where $\pp_v$ is the prime ideal associated to $v$.

Let us start by defining some notation arising from global class field theory. Let $K$ be a global field, and let $S$ be a finite set of primes of $K$. Then we define $I^S$ to be the group of fractional ideals of $K$ whose factorizations do not contain primes from $S$. Let $\mm = \mm_0 \mm_{\infty}$ be a modulus of $K$, where $\mm_0$ denotes the finite part, and $\mm_{\infty}$ denotes the infinite part.

Define
\begin{eqnarray*}
K_{\mm,1} := \{a \in K^{\times} \mid && v(a-1) \geq v(\mm) \textup{ for all finite } v \textup{ dividing } \mm, \textup{ and}\\
&& \textup{the image of } a \textup{ in } K_{v}^{\times} \textup{ is positive for all real } v \textup{ dividing } \mm\}.
\end{eqnarray*}
Then we have a well-defined map
\begin{eqnarray*}
i \colon K_{\mm,1} &\to & I^{S(\mm)}\\
a &\mapsto & (a)
\end{eqnarray*}
where $S(\mm)$ denotes the set of finite primes appearing in the modulus $\mm$. We also define the \textit{ray class group modulo $\mm$} by
$$C_{\mm} = I^{S(\mm)}/i(K_{\mm,1}).$$

\begin{example} If $K = \Q$, and $\mm = 2 \cdot \infty$, then the ray classes modulo $\mm$ give exactly the partition of $K_{\mm,1} = \Z_{(2)}^{\times}$ appearing in \cite{Koe10}, page 7. Namely, these classes are $k + 8\Z_{(2)}$ for $k = 1,3,5,7$.
\end{example}

For a finite abelian extension $L/K$ and a set $S$ of primes of $K$ containing all the primes ramifying in $L$, we also have the global Artin homomorphism
\begin{eqnarray*}
\psi_{L/K} \colon I^S &\to & \Gal(L/K)\\
\pp &\mapsto & (\pp, L/K)
\end{eqnarray*}
where $(\pp, L/K)$ denotes the Frobenius automorphism corresponding to the prime ideal $\pp$. This definition can be linearly extended to all of $I^S$.

\begin{example}
\label{Ex: rayclass}
For any number field $K$, let us consider the extension $L/K$, with $L = K(\sqrt a)$ for $a \in K^{\times}\backslash K^{{\times}2}$. Let $S$ to be the set of primes of $K$ ramifying in this extension. We identify $\Gal(L/K)$ with $\{\pm 1\}$. Let $\mm = 2a \cdot \infty$. To explicitly write down the Artin homomorphism with respect to $\mm$, we want to compute the Frobenius elements of the prime ideals $\pp$ of $K$ for $\pp$ coprime to $(2a)$. In this case, the Frobenius element $(\pp, L/K)$ is given by the power residue symbol $\left(\frac{a}{\pp}\right)$. For a more detailed discussion of the power residue symbol, see Exercise 1.5 of \cite{CasFro86}.

Then the Artin homomorphism is given explicitly by
\begin{eqnarray*}
\psi_{K/\Q} \colon I^S &\to & \Gal(K/\Q)\\
\pp &\mapsto & \left(\frac{a}{\pp}\right),
\end{eqnarray*}
\end{example}

Further, if we let $I_K$ and $I_L$ denote the groups of fractional ideals in $L$ and $K$ respectively, the relative norm map on the prime ideals $\mathfrak{P}$ of $L$ is defined as
\begin{eqnarray*}
\Nm_{L/K} \colon I_L &\to & I_K\\
\mathfrak{P} & \mapsto & \pp^{f(\mathfrak{P}/\pp)},
\end{eqnarray*}
and extended linearly, where $\mathfrak{P}$ lies above $\pp \subseteq \calO_K$.

In particular, the primes of $K$ that split in $L$ lie in $\Nm_{L/K}(I_L)$. Further, we say that a homomorphism $\phi: I^S \to G$ admits a modulus if there exists a modulus $\mm$ with $S(\mm) = S$ such that $\phi(i(K_{\mm,1})) = 1$.

\begin{thm}[Artin Reciprocity] Let $L$ be a finite abelian extension of $K$, and let $S$ be the set of primes of $K$ ramifying in $L$. Then the Artin map $\psi: I^S \to \Gal(L/K)$ admits a modulus $\mm$ with $S(\mm) = S,$ and it defines an isomorphism
$$I_K^{S}/i(K_{\mm,1}) \cdot \Nm(I_L^{S'}) \to \Gal(L/K),$$
where $S'$ denotes the set of primes of $L$ lying over a prime in $S(\mm)$.
\end{thm}

Thus, only the ray classes containing the primes of $K$ that split in $L$ have trivial image under this isomorphism. 

\begin{remark}
For a quadratic extension $\Q(\sqrt m)/\Q$, $\mm = 4m \cdot \infty$ is an admissible modulus, so if $L = \Q(\sqrt{-1}, \sqrt 2)/\Q$, we can take $\mm = 8 \cdot \infty$. From Example \ref{Ex: rayclass}, we see that the splitting behaviour of a prime in $\Q(\sqrt{-1},\sqrt 2)/\Q$ depends on its ray class modulo $\mm$, characterized by $k + 8\Z_{(2)}$, for $k = 1,3,5,7$.
\end{remark}

\subsection{Prescribing Hilbert Symbols}

The main result of this section is Theorem \ref{tate}, which was communicated by Tate \cite{Tat11}, generalizing \cite{Ser73}, Theorem 4, page 24 to the case of any global field and any norm residue symbol.

Let $n > 1$ be an integer, and let $K$ be a global field containing the $n$-th roots of unity, where $\Char K \nmid n$. Let $J$ be the \idele group of $K$, $I_v = K_v^{\times}/K_v^{\times n}$, and $U_v$ denotes the image of units of $\calO_v \subseteq K_v$ in $I_v$. Also, let $I = \prod_{v} '(I_v, U_v/U_v^n) = J/J^n,$ where $\prod_{v}'$ denotes the restricted direct product. Let $P$ be the image of $K^{\times}$ in $I$. 

\begin{prop}
\label{generaltate}
Let $A$ be a finitely generated subgroup of $P$, and for each $v$, let $A_v$ be its image in $I_v$. Then a character of $\prod_v A_v$ which is trivial on $A$ can be extended to a character of $I$ which is trivial on $P$.
\end{prop}
\begin{proof}
We need to show that the natural restriction map between groups of continuous homomorphisms $\Hom(I/P, \mu_n) \to \Hom(\prod_vA_v/A, \mu_n)$ is surjective. This is equivalent to showing the injectivity of $\prod_vA_v/A \to I/P$. For this, it suffices to show that $P \cap \prod_v A_v = A$. The right-to-left inclusion is clear by construction. To show the left-to-right inclusion, let $\alpha \in K^{\times}$ be an element such that if we view it as an element of $P$, then $\alpha_v \in A_v$ for all $v$. Then $K(A^{1/n}, \alpha^{1/n})$ is an extension of $K(A^{1/n})$ which splits at every place. Hence the two fields are equal, and by Kummer theory, this means $\alpha \in A$, as required.
\end{proof}

\begin{lemma}
\label{L: CFT}
For a global field $K$ containing $n$-th roots of unity with $\Char K \nmid n$, the homomorphism
\begin{align*}
I/P &\to \Hom(K^{\times}/K^{\times n}, \mu_n)\\
(b_v)_v &\mapsto (x \mapsto \prod_v(b_v,x)_v)
\end{align*}
is an isomorphism.
\end{lemma}
\begin{proof}
Let $C_K$ denote the \idele class group of $K$. Then $C_K = J/K^{\times}$, and $I/P = C_K/C_K^n$. 

First assume that $K$ is a number field. Using class field theory, there is a surjective Artin homomorphism
$$\psi_K \colon C_K \to \Gal(K^{\textup{ab}}/K),$$
which gives an isomorphism
$$C_K/\ker(\psi_K) \cong \Gal(K^{\textup{ab}}/K).$$
Note that $\ker{\psi_K}$ can be described as the connected component of $1$. Equivalently, this is the image in $C_K$ of the product over the archimedean primes $v$ of $K_v^+$, which is the connected component of $1$ in $K_v$. This means $K_v^+ = \C^{\times}$ or $K_v^+ = \R_{>0}$, depending on whether $v$ is a real or a complex place. Since $C_K^n$ contains $\ker{\psi_K}$, taking the quotient of the Artin homomorphism by $C_K^n$ gives the isomorphism
$$C_K/C_K^n \cong \Gal(K^{\textup{ab}}/K)/n \Gal(K^{\textup{ab}}/K) = \Gal(K^{\textup{ab}}/K)^{\textup{exp } n}.$$

Now suppose that $K$ is a global function field. In this case, we have an Artin homomorphism $\psi_K \colon C_K \to \Gal(K^{\textup{ab}}/K)$, which induces an isomorphism
$$\hat{\psi}_K: \hat{C}_K \to \Gal(K^{\textup{ab}}/K),$$
where $\hat{C}_K$ denotes the profinite completion of the group $C_K$. Using the argument from the above paragraph, we get the isomorphism
$$C_K/C_K^n \cong \hat{C}_K/\hat{C}_K^n \cong \Gal(K^{\textup{ab}}/K)^{\exp n}.$$

Then for any global field $K$, by Kummer theory, there is a perfect pairing
$$K^{\times}/K^{\times n} \times \Gal(K^{\textup{ab}}/K)^{\textup{exp } n} \to  \mu_n.$$

This gives the desired isomorphism
$$\Gal(K^{\textup{ab}}/K)^{\textup{exp } n} \cong \Hom(K^{\times}/K^{\times n}, \mu_n).$$
\end{proof}

Proposition \ref{generaltate} implies a statement analogous to \cite{Ser73}, Theorem 4, page 24:

\begin{thm}
\label{tate}
Let $K$ be a global field. Let $V$ be the set of places of $K$, and let $\Lambda$ be a finite set of indices. Let $(a_i)_{i \in \Lambda}$ be a finite family of elements  in $K^{\times}$ and let $(\vareps_{i,v})_{i \in \Lambda, v \in V}$ be a family of numbers equal to $\pm 1$. In order that there exists $x \in K^{\times}$ such that $(a_i,x)_v = \vareps_{i,v}$ for all $i \in \Lambda$ and $v \in V$, it is necessary and sufficient that the following conditions be satisfied:
\begin{itemize}
\item[(1)] All but finitely many of the $\vareps_{i,v}$ are equal to $1$.
\item[(2)] For all $i \in \Lambda$, we have $\prod_{v \in V} \vareps_{i,v} = 1$.
\item[(3)] For all $v \in V$, there exists $x_v \in K^{\times}$ such that $(a_i,x_v)_v = \vareps_{i,v}$ for all $i \in \Lambda$.
\end{itemize}
\end{thm}
\begin{proof}
Let $A$ be the group generated by the images of the $a_i$ in $I$, and let $A_v$ be the image of $A$ in $I_v$ for each $v$. Let $\vareps_v$ be a character on $A_v$ defined by
\begin{eqnarray*}
\vareps_v \colon A_v &\to & \mu_n\\
a_i &\mapsto & \vareps_{i,v}.
\end{eqnarray*}
By (3), $\vareps_v$ is indeed a character. Then define
\begin{align*}
\chi \colon \prod_{v}A_v &\to \mu_n\\
(a_v)_v &\mapsto \prod_v \vareps_v(a_v) 
\end{align*}

 Since almost all $\vareps_{i,v}$ are $1$ by (1), this gives a well-defined character on $\prod_{v}A_v$. By (2), $\chi$ is trivial on $A$.

Then by Proposition \ref{generaltate},  $\chi$ can be extended to a character $\widetilde{\chi}: I \to \mu_n$ that is trivial on $P$. By Lemma \ref{L: CFT}, this corresponds naturally to an element $x$ of $K^{\times}/K^{\times n}$, so that we can write $\widetilde{\chi}((b_v)_v) = \prod_v(b_v,x)_v$. This gives the conclusion of the theorem, since now we can uniformly write
$$\vareps_{i,v} = \vareps_v(a_i) = (a_i,x)_v$$
for this $x$.
\end{proof}


\subsection{Uniform definition of the ring of integers as intersection of localization rings}

Let $a,b$ be totally positive elements of $K^{\times}$ whose images in $K^{\times}/K^{{\times}2}$ are independent. Then we have in particular that $\Gal(K(\sqrt a, \sqrt b)/K) = \{\pm 1\} \times \{\pm 1\}$, and further, $\sqrt{ab} \notin K$. We would like to see how primes split in the extensions $K(\sqrt{a})/K$ and $K(\sqrt b)/K$ (and hence in $K(\sqrt{a}, \sqrt b)/K)$. This will give us some information about the Hilbert symbols $(a,p)_{\pp}$ and $ (b,p)_{\pp}$. More precisely, let
$$\psi: C_{\mm} \to \Gal(K(\sqrt a, \sqrt b)/K) = \{\pm 1\} \times \{\pm 1\}$$
be the Artin map. Then:

\begin{lemma}
\label{L: artin}
Take $a,b \in K^{\times}$ as above, and let $p \in K^{\times}$. Let $\mm$ be an admissible modulus, corresponding to the extension $K(\sqrt a, \sqrt b)/K$. Further, suppose that $\mm$ is divisible by all primes dividing $2ab$, and we also assume that $\mm$ contains all real places. For a prime $\pp$ in $K$ such that $\pp \nmid \mm_0$, $\pp \in \Delta_{a,p} \cap \Delta_{b,p}$ if and only if $v_{\pp}(p)$ is odd and $\psi(\pp) = (-1,-1)$.

\end{lemma}
\begin{proof}
For any prime ideal $\pp$ of $K$ prime to the modulus $\mm$, $a$ is a $\pp$-adic unit, so we use Equation \ref{E: serre} to compute $(a,p)_{\pp}$. That is, $(a,p)_{\pp} = -1$ if and only if $v_{\pp}(p)$ is odd and $a$ is not a square in the residue field modulo $\pp$. But we have observed that these conditions are equivalent to insisting that $\psi(\pp) = (-1,\pm 1)$. A similar argument applies to $(b,p)_{\pp}$, and we get our conclusion.
\end{proof}

Let us partition the primes of $K$ by their images under $\psi$:
$$\PP^{[i,j]} = \{\textup{prime ideals } \pp \textup{ of } K \mid \psi(\pp) = (i,j)\},$$
where $(i,j) \in \Gal(\Q(\sqrt a, \sqrt b)/\Q)$, with $i,j \in \{\pm 1\}$.
We also let
$$\PP^{[i,j]}(p) = \{\textup{primes } \pp \in \PP^{[i,j]} \textup{ with } v_{\pp}(p) \textup{ odd}\}.$$

As long as the image of $(p)$ under the Artin map is nontrivial, the sets $\PP^{[i,j]}(p)$ have a simple description using Hilbert symbols:

\begin{lemma} 
\label{L: identification}
Suppose that $p \in K^{\times}$, and for the fixed modulus $\mm$ from above, suppose that the fractional ideal $(p)$ has no common factors with $\mm$. Then we have the following identification of sets of primes, where the two sets differ at most by the primes dividing the modulus.
\begin{eqnarray*}
\PP^{[-1,-1]}(p) &\leftrightarrow & \Delta_{a,p} \cap \Delta_{b,p}\\
\PP^{[-1,1]}(p) &\leftrightarrow & \Delta_{a,p} \cap \Delta_{ab,p}\\
\PP^{[1,-1]}(p) &\leftrightarrow & \Delta_{b,p} \cap \Delta_{ab,p},
\end{eqnarray*}
\end{lemma}
\begin{proof}
$\PP^{[-1,-1]}(p)$ is easy: By Lemma \ref{L: artin}, it is $\Delta_{a,p} \cap \Delta_{b,p}$, excluding the primes not dividing the modulus. To prove the second (resp. third) equivalence, we express $\PP^{[1,-1]}(p)$ (resp. $\PP^{[-1,1]}(p)$) in a similar way, via the following identification of the Galois groups:
\begin{eqnarray*}
\Gal(K(\sqrt a, \sqrt b)/K) &\cong & \Gal(K(\sqrt {ab}, \sqrt b)/K) (\textup{resp. } \Gal(K(\sqrt{a}, \sqrt{ab})/K))\\
(\sigma_1, \sigma_2) &\mapsto & (\sigma_1 \sigma_2, \sigma_2) (\textup{resp. }(\sigma_1, \sigma_2) \mapsto  (\sigma_1, \sigma_1 \sigma_2))
\end{eqnarray*}
Composing the original Artin map $\psi$ with this isomorphism, we can draw similar conclusions in these cases as in the case of $\PP^{[-1,-1]}(p)$.
\end{proof}




\begin{defn}
\label{D: Rp}
For each $p,q \in K^{\times}$, let
\begin{eqnarray*}
R_p^{[-1,-1]} &=& \bigcap_{\pp \in \Delta_{a,p} \cap \Delta_{b,p}} (\calO_K)_{\pp}\\
R_p^{[1,-1]} &=& \bigcap_{\pp \in \Delta_{ab,p} \cap \Delta_{b,p}} (\calO_K)_{\pp}\\
R_p^{[-1,1]} &=& \bigcap_{\pp \in \Delta_{a,p} \cap \Delta_{ab,p}} (\calO_K)_{\pp}\\
R_{p,q}^{[1,1]} &=& \bigcap_{\pp \in \Delta_{ap,q} \cap \Delta_{bp,q}} (\calO_K)_{\pp}
\end{eqnarray*}
\end{defn}


The $R$'s are existentially defined subrings of $K$ containing $\calO_K$ as long as no archimedean places are involved in the intersection, since for any $a,b,c,d \in K^{\times}$ such that $\sigma \notin \Delta_{a,b} \cap \Delta_{c,d}$ for each archimedean place $\sigma$,
$$T_{a,b}+T_{c,d} = \bigcap_{\pp \in \Delta_{a,b} \cap \Delta_{c,d}} (\calO_K)_{\pp},$$
by Proposition \ref{P: Tab}.

We would now like to express $\calO_K$ in terms of the $R$'s, through the following lemmas:

\begin{lemma}
\label{L: classnumber1}
Let $\pp$ be a prime ideal of $\calO_K$ with $\pp \nmid \mm_0$, and suppose that $\psi(\pp) = (i,j)$ for $(i,j) \neq (1,1)$. Then $\pp \in \PP^{[i,j]}(p)$ for some $p \in K^{\times}$. Hence, there exist $c,d \in K^{\times}$ such that $\pp \in \Delta_{c,p} \cap \Delta_{d,p}$.
\end{lemma}
\begin{proof}
Choosing $p \in \pp - \pp^2$ will suffice, since we will then have $v_{\pp}(p) = 1$.
\end{proof}

\begin{lemma}
\label{L: classnumber2}
For all prime ideals $\pp$ with $\pp \nmid \mm_0$ and satisfying $\psi(\pp) = (1,1)$, we have $\pp \in \Delta_{ap,q} \cap \Delta_{bp,q}$ for some $p,q \in K^{\times}$ such that $q$ is totally positive.
\end{lemma}
\begin{proof}
We can use the arguments exactly as in the proof of Lemma \ref{L: classnumber1} to find $p$ with $v_{\pp}(p)$ odd. Now, by the definition of $\PP(p,q)$, we only need to find a totally positive $q$ such that $v_{\pp}(q)$ is even (for example, 0), and $q$ not a square in $\F_{\pp}$. This can be achieved by weak approximation.
\end{proof}

\begin{cor} 
\label{C: integrality}
We have
$$\calO_K = \bigcap_{\pp|\mm_0}(\calO_K)_{\pp} \cap \bigcap_{p,q \in (K^{\times})^+} (R_p^{[1,1]} \cap R_p^{[1,-1]} \cap R_p^{[-1,1]} \cap R_{p,q}^{[1,1]})$$
where $(K^{\times})^+$ denotes the set of totally positive elements of $K$.
\end{cor}
\begin{proof}
We use the fact that
$$\calO_K = \bigcap_{\pp}(\calO_K)_{\pp}$$
where $\pp$ ranges over all finite places of $\calO_K$. Now we use Lemmas \ref{L: classnumber1} and \ref{L: classnumber2}. Further, the total positivity of $p$ and $q$ guarantees that no infinite places are included in the intersection.
\end{proof}


\subsection{An existential definition of the Jacobson radical}

\begin{lemma}
\label{L:evenvaluation}
$$K^{\times 2} \cdot T_{a,b}^{\times} = \bigcap_{\pp \in \Delta_{a,b}} v_{\pp}^{-1}(2\Z).$$
\end{lemma}
\begin{proof}
Let us first prove the inclusion of the left-hand side. Let $x \in K^2 \cdot T_{a,b}^{\times}$ and let $v \in \Delta_{a,b}$ be a nonarchimedean valuation. Writing $x = y^2 z$ for some $y \in K^{\times}$ and $z \in T_{a,b}^{\times}$, we have $v(x) = 2v(y)+v(z)$. Since $z \in T_{a,b}^{\times} = \bigcap_{\pp \in \Delta_{a,b}}\calO_{\pp}^{\times}$, By Proposition \ref{P: Tab}, for $v \in \Delta_{a,b}$, we must have $v(z) = 0$, so $v(x)$ is even for all $x \in K^{\times 2} \cdot T_{a,b}^{\times}$ and all nonarchimedean $v \in \Delta_{a,b}$.

Conversely, suppose that we are given a nonzero element $q \in K$ whose valuation at each $v \in \Delta_{a,b}$ is even. Since $\Delta_{a,b}$ is finite, we can find $r \in K^{\times}$ whose valuation at each $v \in \Delta_{a,b}$ is $v(q)/2$ by weak approximation. Then $q/r^2$ has valuation $0$ at each $v \in \Delta_{a,b}$, so $q/r^2 \in T_{a,b}^{\times}$ by Proposition \ref{P: Tab}. Therefore, $q \in K^{\times 2} \cdot T_{a,b}^{\times}$.
\end{proof}

Now, for $c \in K^{\times}$, we define
$$I_{a,b}^c := c \cdot K^2 \cdot T_{a,b}^{\times} \cap (1-K^2 \cdot T_{a,b}^{\times}).$$ Then:

\begin{lemma}
\label{L:I_{a,b}^c}
For $c \in K^{\times}$,
\begin{eqnarray*}
I_{a,b}^c = \{y \in K \mid && v(y) \textrm{ is odd and positive for all } v \in \Delta_{a,b} \cap \PP(c), \textrm{ and }\\
&& v(y)\textrm{ and } v(1-y) \textrm{ are even for all } v \in \Delta_{a,b} \backslash \PP(c)\}
\end{eqnarray*}
\end{lemma}
\begin{proof}
By Lemma \ref{L:evenvaluation}, $y \in I_{a,b}^c$ if and only if $v(y/c)$ and $v(1-y)$ are even for all $v \in \Delta_{a,b}$.

For $v \notin \PP(c)$, $v(c)$ is even, so the condition becomes that $v(y)$ and $v(1-y)$ are even.

For $v \in \PP(c)$, $v(c)$ is odd, so the condition becomes that $v(y)$ is odd and $v(1-y)$ is even, which is equivalent to the condition that $v(y)$ is odd and positive, by the ultrametric inequality.
\end{proof}

\begin{defn}
\label{D: JacobsonRadical}
For $a, b \in K^{\times}$, let
$$ J_{a,b} := \bigcap_{\pp \in \Delta_{a,b} \cap (\PP(a) \cup \PP(b))} \pp \calO_{\pp}.$$
\end{defn}

\begin{lemma}
\label{L: alternatedefnofJ}
We have
\begin{eqnarray*}
J_{a,b} = \{0\} \cup \{x \in K^{\times} \mid && \exists y_1, y_2 \in K \textrm{ such that }\\
&& y_1, x-y_1 \in a \cdot K^2 \cdot T_{a,b}^{\times} \cap (1-K^2 \cdot T_{a,b}^{\times})\\
&& y_2, x-y_2 \in b \cdot K^2 \cdot T_{a,b}^{\times} \cap (1-K^2 \cdot T_{a,b}^{\times})\}.
\end{eqnarray*}
\end{lemma}
\begin{proof}
Let $J_{a,b}'$ denote the right-hand side of the equality in the statement of the lemma.

We begin by showing that
\begin{eqnarray}
\label{E: claim}
I_{a,b}^c + I_{a,b}^c &=& \bigcap_{\pp \in \Delta_{a,b} \cap \PP(c)} \pp \calO_{\pp}.
\end{eqnarray}
We first show the inclusion of the left into the right. Take some $z = y_1 + y_2 \in I_{a,b}^c + I_{a,b}^c$, with $y_1, y_2 \in I_{a,b}^c$. Then we want to show that $v(z) > 0$ for all $v \in \Delta_{a,b} \cap \PP(c)$. By Lemma \ref{L:I_{a,b}^c}, $v(y_1), v(y_2)>0$, so that $v(z) = v(y_1+y_2)>0$ as well.

For the reverse inclusion, take $z \in \cap_{\pp \in \Delta_{a,b} \cap \PP(c)}\pp\calO_{\pp}.$ By weak approximation on the valuations $v \in \Delta_{a,b}$, we can find $y_1 \in K$ satisfying $y_1, z-y_1 \in I_{a,b}^c$. This proves the equality in the above claim.

Finally, we observe that $J_{a,b}' = (I_{a,b}^a + I_{a,b}^a) \cap (I_{a,b}^b + I_{a,b}^b)$. Using the above claim \ref{E: claim},
\begin{eqnarray*}
J_{a,b}' &=& \bigcap_{\pp \in \Delta_{a,b} \cap \PP(a)} \pp \calO_{\pp} \cap \bigcap_{\pp \in \Delta_{a,b} \cap \PP(b)} \pp \calO_{\pp}\\
&=& \bigcap_{\pp \in (\Delta_{a,b} \cap \PP(a)) \cup (\Delta_{a,b} \cap \PP(b))} \pp \calO_{\pp}\\
&=& \bigcap_{\pp \in \Delta_{a,b} \cap (\PP(a) \cup \PP(b))} \pp \calO_{\pp}\\
&=& J_{a,b}
\end{eqnarray*}
and the lemma is proven.
\end{proof}

\begin{cor}
\label{C: diophantineJ}
$J_{a,b}$ is diophantine in $K$.
\end{cor}
\begin{proof}
Since $T_{a,b}$ is diophantine, so is $J_{a,b}$ by Lemma \ref{L: alternatedefnofJ}.
\end{proof}


\subsection{More preliminaries}
\begin{lemma}
\label{L: independentsplitting}
We can choose $a,b \in K^{\times}$ so that the following conditions hold: 
\begin{itemize}
\item[(1)] The images of $a$ and $b$ in $K^{\times}/K^{\times 2}$ are independent.
\item[(2)] $a,b \in 1+ 8 \calO_K$.
\item[(3)] Given an ideal class of $K$ and and element $\sigma \in \Gal(K(\sqrt a, \sqrt b)/K)$, there exists a prime $\qq$ of $K$ in the ideal class such that $\qq \in I^{S(\mm)}$ and $\psi(\qq) = \sigma$. 
\item[(4)] No prime ideal appears in the factorizations of both $(a)$ and $(b)$.
\item[(5)] $a$ and $b$ are totally positive.
\end{itemize}
\end{lemma}

In particular, we deduce from the lemma that this choice of $a$ and $b$ implies that $a$ and $b$ are $\pp$-adic squares for all $\pp|2$, by Hensel's lemma. Further, the splitting behaviour of primes in the extension $K(\sqrt a, \sqrt b)/K$ is independent of its image under the Artin map.

\begin{proof}
Let $H$ denote the Hilbert class field of $K$. Choose a prime ideal $\pp$ not dividing $2$ in $K$. There is some totally positive $a \in K$, with $a \in 1 + 8 \calO_K$ and $v_{\pp}(a) = 1$. Then $K(\sqrt a)/K$ is ramified at $\pp$, so $\sqrt a \notin H$. 

Now, choose another prime ideal $\pp'$ different from $\pp$ and not dividing $2$ in $K$. Again, we choose a totally positive $b \in K$, satisfying $b \in 1 + 8 \calO_K$ and $(a,b) = 1$. Then $\pp'$ ramifies in $K(\sqrt b)$, so $\sqrt b \notin H$. Further, since they ramify in different places, $\sqrt{ab} \notin H$ as well.

Then the field $K(\sqrt a,\sqrt b)$ is linearly disjoint from $H$ over $K$, because of the ramification of $\pp$ and $\pp'$.  Thus, by the Chebotarev density
theorem, we can find a prime ideal $\qq$ of $K$  whose Frobenius element in $\Gal(K(\sqrt a,\sqrt b)/K)$ is $\sigma$ and whose Frobenius element in $\Gal(H/K)$ is prescribed by the given ideal class. Because of the latter, $\qq$ belongs to the given ideal class as in (3).
\end{proof}

\begin{cor}
\label{C: nobadprimes}
Choosing $a$ and $b$ as above, 
\begin{eqnarray*}
\PP^{[(-1,-1)]}(p) &=& \Delta_{a,p} \cap \Delta_{b,p}\\
\PP^{[(-1,1)]}(p) &=& \Delta_{a,p} \cap \Delta_{ab,p}\\
\PP^{[(1,-1)]}(p) &=& \Delta_{b,p} \cap \Delta_{ab,p}.
\end{eqnarray*}
\end{cor}
\begin{proof}
From Lemma \ref{L: identification}, we already know that $\PP^{[(-1,-1)]}(p)$ and $\Delta_{a,p} \cap \Delta_{b,p}$ agree on the primes not dividing $\mm$. So we just need to check that they agree on the primes dividing $\mm$. Since $a$ and $b$ are totally positive, no archimedean primes appear on either side of the equality. So take a prime $\pp$ with $\pp \nmid \mm_0$. If $\pp|2$, then $\pp \notin \Delta_{a,p} \cap \Delta_{b,p}$ since $a,b \in 1 + 8 \calO_K$, so by Hensel's lemma, they are $\pp$-adic squares. So we suppose that $\pp \nmid 2$. Since $p$ is a $\pp$-adic unit, if $(a,p)_{\pp} = -1$, then $v_{\pp}(a)$ must be odd. Similarly, if $(b,p)_{\pp} = -1$, then $v_{\pp}(b)$ must be odd. But we cannot have both $v_{\pp}(a)$ and $v_{\pp}(b)$ odd, since $(a)$ and $(b)$ are relatively prime. Hence, $\pp \notin \Delta_{a,p} \cap \Delta_{b,p}$, and the two sets agree exactly. The proofs of the other two statements are similar.
\end{proof}

Throughout the rest of the paper, we assume that $a$ and $b$ are fixed so that they satisfy Lemma \ref{L: independentsplitting}.


\subsection{Imposing integrality at each finite place}

\begin{defn}
\label{D: phisigma}
For each $\sigma \in \Gal(K(\sqrt a, \sqrt b)/K)$, let
$$\Phi_{\sigma} = \{p \in K^{\times} \mid (p) \in I^{S(\mm)}, \psi((p)) = \sigma, \textrm{ and } \PP(p) \subseteq \PP^{[1,1]} \cup \PP^{[\sigma]}\}.$$
\end{defn}

\begin{lemma}
\label{L: diophantine1}
\hfill
\begin{itemize}
\item[(a)] For each $\sigma \in \Gal(K(\sqrt a, \sqrt b)/K),$ the set $\Phi_{\sigma}$ is diophantine in $K$.
\item[(b)] For any $p \in \Phi_{\sigma}$ and $\sigma \in \Gal(K(\sqrt a, \sqrt b)/K)$ with $\sigma \neq (1,1)$, $\PP^{[\sigma]}(p) \neq \emptyset$. Furthermore, the Jacobson radical of $R_p^{[\sigma]}$, denoted $J(R_p^{[\sigma]})$, is diophantine in $K$.
\item[(c)] For $\sigma \in \Gal(K(\sqrt a, \sqrt b)/K)$ with $\sigma \neq (1,1)$, if $\pp \nmid \mm_0$ is a prime ideal of $K$ satisfying $\psi(\pp) = \sigma$, then there exists $p \in \Phi_{\sigma}$ such that $\pp \in \PP^{[\sigma]}(p)$.
\end{itemize}
\end{lemma}
\begin{proof}
\hfill
\begin{itemize}
\item[(a)] Let us first show that the property that the set $\{p \in K^{\times} \mid (p) \in I^{S(\mm)}, \psi((p)) = \sigma\}$ is Diophantine. To do this, we start by recalling that the trivial ray class consists of ideals in $i(K_{\mm,1})$, where $i: K_{\mm,1} \to I^{S(\mm)}$ is the (well-defined) inclusion given by $p \mapsto (p)$. And we recall that $K_{\mm,1}$ is defined by a finite number of local conditions of the form
$$\ord_{\pp}(a-1) \geq m(\pp),$$
and the total positivity conditions. Since sets in $K$ with constraints of above form are Diophantine by \cite{PooShl05}, Proposition 2.2, $K_{\mm,1}$ is Diophantine.

The other ray classes consisting of principal ideals are a translate of the trivial ray class by some element $p \in K^{\times}$, with $(p)$ being an element of this ray class. Hence, if we denote this ray class by $R$, the set $\{p \in K^{\times} \mid (p) \in R\}$ is also Diophantine. By the finiteness of the ray class number, $\Phi_{\sigma}$ is also Diophantine in $K$.

Now, we need to show that the second condition $\PP(p) \subseteq \PP^{[1,1]} \cup \PP^{[\sigma]}$ is Diophantine. For all $\sigma \in \Gal(K(\sqrt a, \sqrt b)/K)$ with $\sigma \neq (1,1)$, by Corollary \ref{C: nobadprimes},
$$\PP^{[\sigma]}(p) = \emptyset \Leftrightarrow p \in (K^{\times})^2 \cdot (R_p^{[\sigma]})^{\times}.$$
Hence the condition $\PP^{[\sigma]}(p) = \emptyset$ is Diophantine. Since $\PP(p) \subseteq \PP^{[1,1]} \cup \PP^{[-1,-1]}$ is equivalent to $\PP^{[-1,1]}(p) = \PP^{[1,-1]}(p) = \emptyset$, the condition $\PP(p) \subseteq \PP^{[1,1]} \cup \PP^{[-1,-1]}$ is Diophantine. Similar statements hold for other $\sigma \in \Gal(K(\sqrt a, \sqrt b)/K)$.

Since $\Phi_{\sigma}$ is given as the intersection of the two conditions, it is also Diophantine.

\item[(b)] For $\sigma = (-1,-1), (1,-1), (-1,1)$, the hypothesis $p \in \Phi_{\sigma}$ implies that $\PP^{[\sigma]}(p) \neq \emptyset$, since if $p = \prod_{\pp} \pp^{e_{\pp}}$, then
$$\sigma = \psi((p)) = \prod_{\pp} \psi(\pp)^{e_{\pp}}$$
and $\psi(\pp) = (1,1)$ or $\psi(\pp) = \sigma$.
Then we have, for example, $J(R_p^{[-1,-1]}) = J_{a,p} + J_{b,p}$ from Definition \ref{D: JacobsonRadical} and Corollary \ref{C: nobadprimes}. Then by Corollary \ref{C: diophantineJ}, $J(R_p^{[\sigma]})$ is Diophantine. 
From this, the assertion follows from the definition of $R_p^{[\sigma]}$.

\item[(c)] Suppose that $\pp \nmid \mm_0$ is a prime ideal of $K$ satisfying $\psi(\pp) = \sigma$. By Lemma \ref{L: independentsplitting}, we can choose a prime ideal $\qq$ whose ideal class can be represented by $\pp^{-1}$, and satisfies $\psi(\qq) = (1,1)$. Then $\pp \qq = (p)$ for some $p \in K^{\times}$, and $\pp \in \PP(p)$. Further, $\psi(p) = \psi(\pp) \psi(\qq) = \sigma$. So $\pp \in \PP^{[\sigma]}(p)$, as desired.
\end{itemize}
\end{proof}

It remains to obtain analogous statements to Lemma \ref{L: diophantine1} in the case of $\sigma = (1,1)$.

\begin{lemma}
\label{L: choosingq}
Let $\mm$ be a fixed modulus for a number field $K$, and let $\pp_0$ be a prime ideal such that $\pp_0 \nmid \mm_0$. Then there exists infinitely many principal ideals $(q)$, with its generator $q \in K^{\times}$ satisfying
\begin{itemize}
\item[(A)] $\psi((q)) = (-1,-1)$;
\item[(B)] $\left(\frac{q}{\pp_0}\right) = -1$;
\item[(C)] $(q)$ is a prime ideal.
\end{itemize}
\end{lemma}
\begin{proof}
We begin by first showing that there exists an $x' \in K^{\times}$ satisfying $\psi((x')) = (-1,-1)$ and $\left(\frac{x'}{\pp_0}\right) = -1$. Let
$$K_{\mm'} := K^{S(\mm')} = \{\alpha \in K^{\times} \mid \ord_{\pp}(\alpha) = 0 \textup{ for all } \pp|\mm'\}.$$
For any modulus $\mm'$ of $K$, there is a canonical isomorphism
$$K_{\mm'}/K_{\mm',1} \simeq \prod_{\stackrel{\pp|\mm'}{\pp \textup{ real}}}\{\pm\} \times (\calO_K/\mm_0')^{\times},$$

Since $\mm$ is a modulus for $K$, so is $\mm \pp_0$. Then in particular, we have the canonical isomorphisms
\begin{align}
K_{\mm}/K_{\mm,1} &\simeq  \prod_{\stackrel{\pp|\mm}{\pp \textup{ real}}}\{\pm\} \times (\calO_K/\mm_0)^{\times} \textup{ and }\\
K_{\mm \pp_0}/K_{\mm \pp_0,1} &\simeq  \prod_{\stackrel{\pp|\mm \pp_0}{\pp \textup{ real}}}\{\pm\} \times (\calO_K/\mm_0 \pp_0)^{\times}\\
\label{E: isomorphism}
&\simeq  K_{\mm}/K_{\mm,1} \times (\calO_K/\pp_0)^{\times},
\end{align}
where the last isomorphism follows by the Chinese remainder theorem, since $\pp_0 \nmid \mm$. Since $K_{\mm}/K_{\mm,1}$ surjects onto the group of ray classes of principal ideals modulo $\mm$, each element of $K_{\mm}/K_{\mm,1}$ determines a ray class of principal ideals modulo $\mm$. Similarly, each element of $K_{\mm \pp_0}/K_{\mm \pp_0,1}$ determines a ray class of principal ideals modulo $\mm \pp_0$.

By Lemma \ref{L: independentsplitting} (3), there exists a principal ideal $(x_1)$ such that $\psi((x_1)) = (-1,-1)$. Then choose a ray class $x'' \in K_{\mm\pp_0}/K_{\mm\pp_0,1}$ mapping under \ref{E: isomorphism} to the class of $x_1$ in $K_{\mm}/K_{\mm,1}$ and to a nonsquare in $(\calO_K/\pp_0)^{\times}$. Then this mod-$\mm\pp_0$ ray class consists of principal ideals $(x')$ satisfying $\psi((x')) = (-1,-1)$ and $\left(\frac{x'}{\pp_0}\right)=-1$. By the Chebotarev density theorem, there are infinitely many prime ideals in the same ray class.  For any such ideal $\qq$, since $\qq$ maps to an element of $K_{\mm}/K_{\mm,1}$, it is principal, say $\qq = (q)$.  This $q$ satisfies (A), (B), and (C).
\end{proof}

\begin{defn}
For $\sigma \in \Gal(K(\sqrt a, \sqrt b)/K)$, define
\begin{eqnarray*}
\widetilde{\Phi_{\sigma}} &:=& K^{{\times}2} \cdot \Phi_{\sigma}\\
\Psi &:=& \left\{(p,q) \in \widetilde{\Phi_{(1,1)}} \times \widetilde{\Phi_{(-1,-1)}} \textup{ }\Bigg|\textup{ } \prod_{\pp | \mm}(ap,q)_{\pp} = -1 \textup{ and } p \in a \cdot K^{\times 2} \cdot (1+J(R_q^{[-1,-1]}))\right\}.
\end{eqnarray*}
\end{defn}

\begin{lemma}
\label{L: diophantine2}
\hfill
\begin{itemize}
\item[(a)] $\Psi$ is Diophantine in $K$.
\item[(b)] For $(p,q) \in \Psi$, we have $\Delta_{ap,q} \cap \Delta_{bp,q} \cap I^{S(\mm)} \neq \emptyset$, and consequently, $J(R_{p,q}^{[1,1]})$ is Diophantine in $K$.
\item[(c)] For each prime ideal $\pp_0$ satisfying $\pp_0 \nmid \mm$ and $\psi(\pp_0) = (1,1)$, there exists $(p,q) \in \Psi$ such that $\Delta_{ap,q} \cap \Delta_{bp,q} = \{\pp_0\}$.
\end{itemize}
\end{lemma}
\begin{proof}
\hfill
\begin{itemize}
\item[(a)] By Lemma \ref{L: diophantine1}(a), $\Phi_{\sigma}$ is Diophantine, so $\widetilde{\Phi_{\sigma}}$ is Diophantine. The fact that $\Psi$ is Diophantine follows from the previous sentence, Lemma \ref{L: diophantine1}(b) and the fact that the condition $\prod_{\pp|\mm}(ap,q)_{\pp} = -1$ consists of finitely many local conditions and hence it cuts out a Diophantine set from $K^{\times} \times K^{\times}$, by \cite{PooShl05}, Proposition 2.2.

\item[(b)] Suppose that $(p,q) \in \Psi$. By Hilbert reciprocity, there is at least one prime ideal $\pp \nmid \mm$ such that $(ap,q)_{\pp} = -1$. Then either $v_{\pp}(ap)$ or $v_{\pp}(q)$ is odd by equation \ref{E: serre}. But $v_{\pp}(a) = 0$ and $p \in \widetilde{\Phi_{(1,1)}}$ and $q \in \widetilde{\Phi_{(-1,-1)}}$, so $\pp \in \PP(p) \cup \PP(q) \subseteq \PP^{[1,1]} \cup \PP^{[-1,-1]}$. 

We claim that $\pp \in \PP^{[1,1]}$. Suppose otherwise, that is, $\pp \in \PP^{[-1,-1]}$. Then in particular, $\pp \notin \PP(p)$, so $v_{\pp}(ap)$ is even. So $v_{\pp}(q)$ must be odd. Therefore, $\pp \in \PP^{[-1,-1]}(q)$, which means that $R_q^{[-1,-1]} \subset \calO_{\pp}$. Then we have by Definition \ref{D: JacobsonRadical} that $J(R_q^{[-1,-1]}) \subset \pp \calO_{\pp}$, which implies by Hensel's lemma that $1 + J(R_q^{[(-1,-1)]}) \subset \calO_{\pp}^{\times 2}$. From $p \in a \cdot K^{\times 2} \cdot (1 + J(R_q^{[-1,-1]})$, we deduce that $ap \in K_{\pp}^{\times 2}$, which implies that $(ap,q)_{\pp} = 1$, a contradiction. Hence $\pp \in \PP^{[1,1]}$, as claimed.

On the other hand, if $v_{\pp}(q)$ is even, then $(a,q)_{\pp} = (b,q)_{\pp} = 1$. If $v_{\pp}(q)$ is odd, then by Lemma \ref{L: artin}, $(a,q)_{\pp} = (b,q)_{\pp} = -1$. In either case, $(a,q)_{\pp} = (b,q)_{\pp}$, so $(bp,q)_{\pp} = (ap,q)_{\pp} = -1$, so $\pp \in \Delta_{ap,q} \cap \Delta_{bp,q}$. Since we had $\pp \nmid \mm$, in fact $\pp \in \Delta_{ap,q} \cap \Delta_{bp,q} \cap I^{S(\mm)}$. Hence, $\Delta_{ap,q} \cap \Delta_{bp,q} \cap I^{S(\mm)} \neq \emptyset$. Then $J(R_{p,q}^{[1,1]}) = J_{ap,q} + J_{bp,q}$, so $J(R_{p,q}^{[1,1]})$ is Diophantine by Corollary \ref{C: diophantineJ}.

\item[(c)] Fix a prime $\pp_0 \nmid \mm$ satisfying $\psi(\pp_0) = (1,1)$. We would like to find $(p,q) \in \Psi$ such that $\pp_0 \in \Delta_{ap,q} \cap \Delta_{bp,q}$.

For each $\pp|\mm_0$, let $E_{\pp}$ be a finite subset of $K^{\times}$ whose image in $K_{\pp}^{\times}/K_{\pp}^{\times 2}$ is a basis. By Chinese remainder theorem, we choose each $e \in \bigcup_{\pp|\mm_0} E_{\pp}$ so that $e \equiv 1 \pmod{\pp_0}$. Since $\bigcup_{\pp|\mm_0} E_{\pp}$ is finite, there are only finitely many prime ideals that are generated by elements in it.

By Lemma \ref{L: choosingq}, there are infinitely many ideals $(q)$ with $q \in K^{\times}$ satisfying:
\begin{itemize}
\item[(A)] $\psi((q)) = (-1,-1)$;
\item[(B)] $\left(\frac{q}{\pp_0}\right) = -1$;
\item[(C)] $(q)$ is a prime ideal.
\end{itemize}
Choose an ideal $(q)$ that is not generated by elements in $\bigcup_{\pp|\mm_0} E_{\pp}$. We note that this choice of $q$ implies that
\begin{eqnarray}
\label{E: choiceofq}
\Delta_{a,q} \cap \Delta_{b,q} = \{(q)\}.
\end{eqnarray}
By (A), $q \in \Phi_{(-1,-1)} \subseteq \widetilde{\Phi_{(-1,-1)}}$. Also, we choose $e_0$ so that $\left(\frac{e_0}{(q)}\right) = -1$ and $\left(\frac{e_0}{\pp_0}\right) = 1$. From $\left(\frac{e_0}{(q)}\right) = -1$, the image of $\{q, e_0\} \subseteq K^{\times}$ is a basis for $K_{(q)}/K_{(q)}^{\times 2}$. Further, $v_{(q)}(e_0) = 0$, and $v_{\pp_0}(e_0)$ is even.


We claim that there exists $p \in K^{\times}$ satisfying the following constraints. 

\begin{itemize}
\item[(1)] $(e,p)_{\pp} = 1$ for all $\pp | \mm_0$ and $e \in E_{\pp}$.
\item[(2)] $(e_0,p)_{(q)} = 1$ and $(q,p)_{(q)} = -1$. 
\item[(3)] $(a,p)_{\pp} = 1$ and $(b,p)_{\pp} = 1$
for each $\pp \nmid \mm$.
\item[(4)] $(q,p)_{\pp_0} = -1$.
\item[(5)] $\prod_{\pp|\mm}(ap,q)_{\pp} = -1$.
\item[(6)] For all archimedean places $\qq$, $(p,q)_{\qq} = 1$.
\end{itemize}

We will use Theorem \ref{tate} to choose such a $p$ according to the prescription of the Hilbert symbol given by the following table. The columns are indexed by the places $v$, the rows are indexed by the $a_i$ to be used in Theorem \ref{tate}, and the entries in the table are the $\varepsilon_{i,v}$. 

\begin{center}
\begin{tabular}{| c || c | c | c |}
	\hline    
    & $\pp_0$ & $(q)$ & all other places \\
    \hline \hline
    any $e \in E_{\pp}$ for $\pp | \mm_0$ & $1$ & $1$ & $1$\\
    \hline
    $e_0$ & $1$ & $1$ & $1$\\
    \hline
    $q$ & $-1$ & $-1$ & $1$\\
    \hline
    $a$ & $1$ & $1$ & $1$\\
    \hline
    $b$ & $1$ & $1$ & $1$\\
    \hline
\end{tabular}
\end{center}

Almost all entries in the above table are $1$. Also, the product of the entries in each row is $1$. For the existence of a local element satisfying the prescription in the $(q)$-column, we claim that $a$ is one such element: $(a,e)_{(q)} = 1$ for all $e \in \bigcup_{\pp|\mm_0}E_{\pp} \cup \{a,b, e_0\}$, since $a$ and $e$ are $q$-adic units. From Lemma \ref{L: artin}, $(a,q)_{(q)} = -1$. For the existence of a local element satisfying the prescription in the $\pp_0$-column, we claim that any $x \in \pp_0 - \pp_0^2$ works. By equation \ref{E: serre}, $(x,q)_{\pp_0} = -1$, since $q$ is not a square modulo $\pp_0$ by construction. Further, $(x,e)_{\pp_0} = 1$ for $e \in E_{\pp}$ with $\pp|\mm_0$ by equation \ref{E: serre} and the choice of $e$ such that $e \equiv 1 \pmod{\pp_0}$. Also, $(x,a)_{\pp_0} = (x,b)_{\pp_0} = 1$ since $\psi(\pp_0) = (1,1)$, by Lemma \ref{L: artin}. Finally, by equation \ref{E: serre}, $(x,e_0)_{\pp_0} = \left(\frac{e_0}{\pp_0}\right) = 1$, so $x$ satisfies the prescription on the $\pp_0$-column.


Thus, the conditions of Theorem \ref{tate} hold, so there exists $p$ satisfying the above prescription of Hilbert symbols. We now show that this $p$ satisfies the five constraints given earlier in the proof. Constraints (1), (2), (3) and (4) are evident from the definition of $p$: constraint (1) from the first row; constraint (2) from the second and third rows; constraint (3) from the last entries of fourth and fifth rows; and constraint (4) comes from the prescription in the $(q, \pp_0)$-entry. Constraint (5) is automatic, since
\begin{align*}
\prod_{\pp|\mm}(ap,q)_{\pp} &= \prod_{\pp|\mm}(a,q)_{\pp} \textup{ (since } (p,q)_{\pp} = 1 \textup{ for all } \pp|\mm \textup{)}\\
&= \prod_{\pp \nmid \mm}(a,q)_{\pp} \textup{ (Hilbert Reciprocity)}\\
&= (a,q)_{(q)} = -1.
\end{align*}
Finally, constraint (6) is clear from the prescription of the third row.

Since $(e_0,a)_{(q)} = 1$ and $(q,a)_{(q)} = -1$, the first half of (2) implies that $(e_0, ap)_{(q)} = 1$ and $(q, ap)_{(q)} = 1$, so $v_{(q)}(ap)$ is even. Choose $y \in K^{\times}$ such that $v_{(q)}(y) = v_{(q)}(ap)/2$, and let $z = ap/y^2$. Then $v_{(q)}(z) = 0$, and $(q,z)_{(q)} = (e_0, z)_{(q)} = 1$. Since the Hilbert symbol is a nondegenerate pairing, $z$ is a $(q)$-adic square. Thus, $ap = z y^2$ is also a $(q)$-adic square.

Now we will show that $(p,q) \in \Psi$. By the nondegeneracy of the Hilbert symbol as a bilinear pairing on $K_v^{\times}/K_v^{{\times}2} \times K_v^{\times}/K_v^{{\times}2} \to \{\pm 1\}$, the first constraint implies that $v_{\pp}(p)$ is even for each $\pp|\mm_0$. Now, using weak approximation, find some $r \in K^{\times}$ satisfying $v_{\pp}(r) = v_{\pp}(p)/2$ for each $\pp|\mm_0$. We may divide $p$ by $r^2$, without changing any of the Hilbert symbols 
involving $p$, so as to assume that $(p) \in I^{S(\mm)}$. We claim that constraint (3) implies that $\psi((p)) = (1,1)$. Write $(p) = \prod_{\pp}\pp^{e_{\pp}}$.
For each prime $\pp \nmid \mm$, we have $(a,p)_{\pp} = (b,p)_{\pp} = 1$. then either $v_{\pp}(p)$ is even, or $a$ and $b$ are squares mod $\pp$. If $e_{\pp}$ is odd, then $\psi(\pp) = (1,1)$ since $a$ and $b$ are squares mod $\pp$. If $e_{\pp}$ is even, then $\psi(\pp)^{e_{\pp}} = (1,1)$. Hence, $\psi((p)) = \prod_{\pp} \psi(\pp)^{e_{\pp}} = (1,1)$. The two conditions in (2) implies that $ap$ is a $(q)$-adic square. Then $p \in a \cdot K^{{\times}2} \cdot (1 + J(R_{q}^{[(-1,-1)]}))$, by the following reason: 
From \ref{E: choiceofq}, $\Delta_{a,q} \cap \Delta_{b,q} = \{(q)\}$. Since $a$ and $b$ were chosen to be $\pp$-adic squares for $\pp|2$, $J(R_q^{[(-1,-1)]}) = q(\calO_K)_{(q)}$. Then $K^{\times 2}(1+J(R_q^{[(-1,-1)]})) = K_{(q)}^{\times 2} \cap K^{\times}$. Since $ap \in K_{(q)}^{\times 2} \cap K^{\times}$, the equality from the previous sentence gives that $p \in a \cdot K^{\times 2} \cdot (1 + J(R_q^{[(-1,-1)]}))$.  Thus, constraints (1),(2),(3),(5) give that $(p,q) \in \Psi$. 

From (4), $(ap,q)_{\pp_0} = (bp,q)_{\pp_0} = -1$. From the second half of (2), $(ap,q)_{(q)} = 1$ since $(a,q)_{(q)} = -1$. For any other place $\qq \nmid \mm$, $(ap,q)_{\qq} = 1$ by the prescription of the Hilbert symbols along the $q$-row. Thus, $\Delta_{ap,q} \cap \Delta_{bp,q} \cap I^{S(\mm)} = \{\pp_0\}$.

Now, if we could show that $\Delta_{ap,q} \cap \Delta_{bp,q}$ contains no primes dividing $\mm$, we would be done. Since $a$ and $b$ were chosen to be totally positive, $(a,q)_{\qq} = (b,q)_{\qq} = 1$ for any archimedean place $\qq$. Further, from the prescription of Hilbert symbols, $(p,q)_{\qq} = 1$ for any archimedean place $\qq$. Thus, no archimedean primes appear in $\Delta_{ap,q} \cap \Delta_{bp,q}$. Now suppose $\qq|2$. Then $(a,q)_{\qq} = 1$ since $a$ is a $2$-adic square, and $(p,q)_{\qq} = 1$ from the construction of $p$. This means $(ap,q)_{\qq} = 1$, so $\qq \notin \Delta_{ap,q} \cap \Delta_{bp,q}$. So suppose that $\qq|\mm$ and $\qq \nmid 2$. Then by constraint (1), $(p,q)_{\qq} = 1$. But then since $(a)$ and $(b)$ are coprime, at most one of $(a,q)_{\qq}$ and $(b,q)_{\qq}$ can be $-1$ by equation \ref{E: serre} since $v_{\qq}(q)$ is even. Hence, we again have $\qq \notin \Delta_{ap,q} \cap \Delta_{bp,q}$. This implies that $\Delta_{ap,q} \cap \Delta_{bp,q} \cap I^{S(\mm)} = \Delta_{ap,q} \cap \Delta_{bp,q} = \{\pp_0\}$. \end{itemize}
\end{proof}



\section{Proof of the main theorem}

For a semilocal subring $R = \bigcap_{\pp \in \Delta} \calO_{\pp}$ of $K$, where $\Delta$ is some finite set of finite places of $K$, we define
$$\widetilde{R} = \{x \in K \mid \not \exists y \in J(R) \textrm{ with } xy = 1\}.$$

\begin{lemma} 
\label{L: universal}
Keeping the notation from above,
\begin{itemize}
\item[(a)] If $J(R)$ is diophantine in $K$, then $\widetilde{R}$ is defined by a universal formula in $K$.
\item[(b)] $\widetilde{R} = \bigcup_{\pp \in \Delta}\calO_{\pp}$, provided that $\Delta \neq \emptyset$ (that is, $R \neq K$).
\end{itemize}
\end{lemma}
\begin{proof}
See \cite{Koe10}, Lemma 14.
\end{proof}

\begin{theorem} 
\label{T: main}
For any number field $K$,
$$\calO_K = \bigcap_{\pp|\mm_0}\widetilde{(\calO_K)_{\pp}} \cap \left(\bigcap_{\sigma \neq (1,1)}\bigcap_{p \in \Phi_{\sigma}} \widetilde{R_p^{\sigma}}\right) \cap \bigcap_{(p,q) \in \Psi} \widetilde{R_{p,q}^{[1,1]}},$$
where $\Phi_{\sigma}$ and $\Psi$ are the diophantine sets defined in the previous section.
\end{theorem}
\begin{proof}
By Lemmas \ref{L: diophantine1}(b) and \ref{L: diophantine2}(b), all the sets $\PP^{[\sigma]}(p)$ and $\Delta_{ap,q} \cap \Delta_{bp,q}$ are nonempty for $p \in \Phi_{\sigma}$ and $(p,q) \in \Psi$. So by Corollary \ref{C: nobadprimes}, Lemma \ref{L: universal}(b) and Definition \ref{D: Rp}, the right-hand side is equal to
$$\bigcap_{\pp|\mm_0}(\calO_K)_{\pp} \cap \left(\bigcap_{\sigma \neq (1,1)}\bigcap_{p \in \Phi_{\sigma}} \bigcup_{\pp \in \PP^{[\sigma]}(p)} (\calO_K)_{\pp}\right) \cap \bigcap_{(p,q) \in \Psi} \bigcup_{\pp \in \Delta_{ap,q} \cap \Delta_{bp,q}}(\calO_K)_{\pp}.$$
Assume first that $\pp_0$ is a prime ideal satisfying $\psi(\pp_0) = \sigma$, where $\sigma \neq (1,1)$. We claim that we can find $p,p' \in \Phi_{\sigma}$ such that 
$$(\calO_K)_{\pp_0} = \bigcup_{\pp \in \PP^{\sigma}(p)}(\calO_K)_{\pp} \cap \bigcup_{\pp \in \PP^{\sigma}(p')}(\calO_K)_{\pp}.$$
First suppose that $\sigma = (-1,-1)$. By Lemma \ref{L: diophantine1}(c), choose a $p \in \Phi_{\sigma}$ such that $\pp_0 \in \PP^{\sigma}(p)$, and let $\pp_1, \cdots, \pp_n$ be the rest of the primes in $\Delta_{a,p} \cap \Delta_{b,p}$. By Lemma \ref{L: independentsplitting}, choose a prime ideal $\qq$ in the ideal class of $\pp_0^{-1}$, with $\psi(\qq) = (1,1)$. Since there are infinitely many choices for $\qq$, we may further assume that $\qq$ is not equal to $\pp_1, \cdots, \pp_n$. Let $(p') = \pp_0 \qq$. Then $p' \in \Phi_{\sigma}$ by construction, and
$$(\Delta_{a,p} \cap \Delta_{b,p}) \cap (\Delta_{a,p'} \cap \Delta_{b,p'}) = \PP^{[-1,-1]}(p) \cap \PP^{[-1,-1]}(p') = \{\pp_0\},$$
where the first equality follows by Corollary \ref{C: nobadprimes}.
So the integrality at $\pp_0$ is imposed.
The arguments for $\sigma = (-1,1)$ and $\sigma = (1,-1)$ are similar to the above argument.

Now, if $\pp_0$ is a prime satisfying $\psi(\pp_0) = (1,1)$, then Lemma \ref{L: diophantine2} lets us choose $(p,q) \in \Psi$ such that $\{\pp_0\} = \Delta_{ap,q} \cap \Delta_{bp,q}$, so that
$$\bigcup_{\pp \in \Delta_{ap,q} \cap \Delta_{bp,q}}(\calO_K)_{\pp} = (\calO_K)_{\pp_0}.$$
Then the integrality at $\pp_0$ is imposed, and the theorem is proven.

\end{proof}

\begin{cor}
For any number field $K$, $\calO_K$ is defined by a first-order universal formula.
\end{cor}
\begin{proof}
By Lemmas \ref{L: diophantine1}(b) and \ref{L: diophantine2}(b), $J(R_p^{\sigma})$ and $J(R_{p,q}^{[1,1]})$ are diophantine, for 
$$\sigma \in \Gal(K(\sqrt a, \sqrt b)/K)$$
with $\sigma \neq (1,1)$. Then by Lemma \ref{L: universal}(a), the right-hand side is defined by a universal formula. By Lemmas \ref{L: diophantine1}(a) and \ref{L: diophantine2}(a), $\Phi_{\sigma}$ and $\Psi$ are diophantine. Hence, the right-hand side appearing in the statement of Theorem \ref{T: main} is defined by a first-order universal formula.
\end{proof}

\section*{Acknowledgements}
I would like to thank my advisor, Bjorn Poonen, for suggesting this problem, for many helpful conversations, and for the comments on improving the exposition. I also thank John Tate for providing an outline of the proof of Theorem \ref{tate}, and for giving his permission to include it in this paper. I am also grateful to Gregory Minton and Bianca Viray for helpful discussions.

\end{document}